\documentclass[reqno]{amsart}
\usepackage[T1]{fontenc}

\usepackage{amssymb}
\usepackage{amsfonts}
\usepackage{amsthm}
\usepackage{amstext}
\usepackage{amsmath}
\usepackage{amscd}
\usepackage[mathscr]{eucal}
\usepackage{url}

\usepackage{hyperref}

\newtheorem{theorem}{Theorem}[section]

\newtheorem{corollary}[theorem]{Corollary}
\newtheorem{lemma}[theorem]{Lemma}

\theoremstyle{definition}

\newtheorem{example}[theorem]{Example}

\theoremstyle{remark}
\newtheorem{remark}[theorem]{Remark}

\newcommand\mL{L\kern-0.08cm\char39}

\begin{document}

\begin{large}

%%% Article information
\title[Specification property for step skew products]{Specification property for step skew products}

%%% Author information
\author[{\v L}. Snoha]{{\v L}ubom\'\i r Snoha}
\address{Department of Mathematics, Faculty of Natural Sciences, Matej
	Bel University, Tajovsk\'eho 40, 974 01 Bansk\'a Bystrica, Slovakia}
\email{lubomir.snoha@umb.sk}

\thanks{The author was supported by the Slovak Research and Development	Agency 
under the contract No.~APVV-15-0439 and by VEGA grant 1/0158/20.}

%    General info


\begin{abstract} Step skew products with interval fibres and a subshift as a base are considered. 
It is proved that if the fibre maps are continuous, piecewise monotone, expanding and surjective
and the subshift has the specification property and a periodic orbit such that the composition 
of the fibre maps along this orbit is mixing, then the corresponding step skew product has the 
specification property.
\end{abstract}

\maketitle

\section{Introduction}
An interesting topic combining nonautonomous dynamical systems, skew products and random dynamical systems is the following. Let $T_1$ and $T_2$ be two continuous selfmaps of the interval $I=[0,1]$. If $x_0\in I$, we decide, by tossing a coin, whether we take $x_1= T_1(x_0)$ or $x_1= T_2(x_0)$. After $(n-1)$ steps arriving at $x_{n-1}$ we again decide, tossing a coin, whether we take $x_n= T_1(x_{n-1})$ or $x_n= T_2(x_{n-1})$. We thus obtain a sequence $(x_n)_{n=0}^\infty$. It can be viewed as the trajectory of the point $x_0$ in the nonautonomous system given by the sequence of maps determined by the coin tossing, each of the maps being either $T_1$ or $T_2$. Since any choice of $\omega = \omega_0\omega_1\omega_2\dots \in \Sigma_2^+ = \{1,2\}^{\mathbb Z_{+}}$, where $\mathbb Z_{+}=\{0,1,2,\dots\}$,
yields a nonautonomous system given by the sequence of maps $T_{\omega_0},T_{\omega_1},T_{\omega_2},\dots$, all such nonautonomous systems are in a sense present in the skew product $(\omega,x) \mapsto (S(\omega), T_{\omega_0}(x))$ where $S$ is the shift transformation $\Sigma^{+}_{2} \to \Sigma^{+}_{2}$, $(S\omega)_n = \omega_{n+1}$.

As a straightforward generalization, one can consider $\Sigma_n^+ = \{1,2, \dots, n\}^{\mathbb Z_{+}}$ and $n$ continuous maps $T_1, T_2, \dots, T_{n}$. Also, instead of the full shift $\Sigma^{+}_{n}$ one can consider a subshift $B\subseteq \Sigma^{+}_{n}$. The present paper deals with the \emph{step skew product} $F\colon B\times I \to B\times I$ defined by 
\begin{equation}\label{Eq:F}
F(\omega,x) = (S(\omega), T_{\omega}(x))
\end{equation}
where $S$ is the shift transformation $\Sigma^{+}_{n} \to \Sigma^{+}_{n}$ and the continuous fibre map $T_{\omega}$ depends only on the first (i.e. beginning) coordinate $\omega_0$ of $\omega$. Clearly, $F$ is continuous.

The dynamics of step skew products with interval fibres has been studied by many authors, usually under additional assumptions on fibre maps, see e.g. the recent papers~\cite{GH2, GO} and references therein.

In the present paper we study the specification property of step skew products \eqref{Eq:F}. The specification property was introduced by Bowen~\cite{Bow}, a chapter on it can be found in~\cite{DGS}. It is a very strong property; systems with this property have a dense subset of periodic points and are topologically mixing. For continuous maps on the interval, specification is equivalent with topological mixing~\cite{bl, bu}. The specification property for step skew products with circle rotations in the fibres was studied in~\cite{MO}.

Our main result is Theorem~\ref{T:main} and Corollary~\ref{C:cor}. However, also  Theorem~\ref{T:non-shrink} used in the proof of Theorem~\ref{T:main} and Example~\ref{Ex:nonaut} are of some interest.

\section{Preliminaries}\label{S:def prelim}

Let $\mathbb N$ be the set of positive integers. Let $(Y, R)$ be a topological dynamical system, which means that $R$ is a~continuous transformation on the compact metric space $Y$ with metric $d$. The dynamical system $(Y, R)$ has the \emph{specification property}, if for every $\varepsilon > 0$ there is an integer $M = M(\varepsilon)$ with the following property: For any integer $k\geq2$, and for any $(y_1, n_1), (y_2, n_2), \ldots, (y_k, n_k) \in Y \times \mathbb{N}$ there exists a point $u \in Y$ with $R^{r_k} (u) = u$ such that 
\[
d(R^i(y_j), R^{r_{j-1}+i}(u))\leq \varepsilon \hskip0.2cm \textrm{for}\hskip0.2cm 0\leq i<n_j \hskip0.2cm\textrm{and}\hskip0.2cm 1\leq j\leq k
\]
where $r_0 = 0$ and $r_j = n_1+n_2+\ldots +n_j+jM$ for $1\leq j\leq k$. We call $M$ the \emph{gap length} for the given~$\varepsilon$.\footnote{Thus, $M=M(\varepsilon)$ is such that for every finite family of orbit segments, if all the gap lengths are \emph{prescribed to be equal to $M$}, an $\varepsilon$-tracing periodic point $u$ does exist. This is equivalent with the definition, also often used, in which  $M=M(\varepsilon)$ is such that for every finite family of orbit segments, if all the gap lengths are \emph{prescribed and greater than or equal to $M$}, an $\varepsilon$-tracing periodic point $u$ still does exist. One implication is trivial. To prove the other one, let a system have the specification property according to the former definition, with $M=M(\varepsilon)$. To show that it has the specification property according to the latter definition, let $\varepsilon>0$ be given. We claim that the same $M=M(\varepsilon)$ works. Indeed, let us have a finite family of orbit segments with lengths $n_i$ and consider any prescribed gap lengths $M+L_i$ with $L_i\geq 0$. We replace the orbit segments of points $y_i$ with lengths $n_i$ by the orbit segments of the same points $y_i$ with lengths $n_i+L_i$. Then, since $M$ was taken from the former definition, an $\varepsilon$-tracing periodic point $u$ exists for this new system of (longer) orbit segments and with all gap lengths equal to $M$. This point $\varepsilon$-traces also the original (shorter) orbit segments, with the gap lengths $M+L_i$, as required.} We also say that the map $R$ itself has the specification property.

Recall that a dynamical system $(X,T)$ or the map $T$ itself is called \emph{(topologically) mixing} if for every pair of nonempty open sets $U$ and $V$ there is a positive integer $N$ such that $T^n(U) \cap V \neq \emptyset$ for all $n \geq N$. It is called \emph{locally eventually onto} (or \emph{topologically exact}) if for every nonempty open set $U$ there is a positive integer $n$ such that $T^n(U) = X$.

By an interval we mean a nondegenerate interval. The length of an interval $J$ is denoted by $|J|$. A continuous map $T\colon [0, 1]\to [0, 1]$ is called \emph{piecewise monotone}, if there is a finite partition $\mathcal{P}$ of $[0, 1]$ into intervals, such that $T{\mid}_P$ is monotone for all $P\in\mathcal{P}$. The endpoints of the intervals in $\mathcal{P}$ are called \emph{critical points}. We say that $T$ is \emph{expanding}, if there is $\alpha > 1$ such that $|T(x)-T(y)| \geq \alpha|x-y|$ holds for all $x, y$ which are in the same element of $\mathcal{P}$. In such a case we call $\alpha$ the \emph{expansion rate}. Recall also that a continuous piecewise monotone map $T\colon [0, 1]\to [0, 1]$ is mixing if and only if it is locally eventually onto, see e.g.~\cite[p. 158]{BC}.

\begin{lemma}\label{L:mix-exp}
	Let $T$ be a mixing, piecewise monotone, continuous transformation on $[0, 1]$. For every $\gamma > 0$ there is an integer $m$ such that $T^m(U)=[0, 1]$ holds for all intervals~$U$ with $|U|\geq \gamma$.
\end{lemma}

\begin{proof}
	Let $\{U_1, \dots, U_k\}$ be a partition of $[0,1]$ into intervals whose lengths are smaller than $\gamma/2$. Since $T$ is locally eventually onto, for $i=1,\dots,k$ there is a positive integer $m_i$ with $f^{m_i}(U_i)=[0,1]$. Put $m:=\max\{m_1, \dots, m_k\}$ and realize that if $|U|\geq \gamma$ then the interval $U$ contains at least one of the intervals $U_i$.
\end{proof}

A sequence $(f_i)^{\infty}_{i=0}$ of maps from $[0, 1]$ to $[0, 1]$ is called a \emph{nonautonomous system}. It is called \emph{finite}, if only finitely many different maps occur in this sequence. For $i\geq1$ and $j\geq0$ we set $f^i_j=f_{j+i-1}\circ\cdots\circ f_{j+1}\circ f_j$ and $f^0_j$ denotes the identity.

Let $\Sigma^{+}_{n} = \{1, 2, \ldots , n\}^\mathbb {Z_{+}}$ be the full $n$-shift with product topology and shift transformation~$S$. Let $\varrho $ be a metric on $\Sigma^{+}_n$ which generates the product topology, such that 
\begin{equation}\label{Eq:metric}
\text{$\varrho(\omega,\eta)<1\,$ implies that $\omega$ and $\eta$ have the same first symbol.}
\end{equation}
We define then a metric on $X= \Sigma^+_n \times [0,1]$ by $d((\omega, x), (\eta, y)) =\max(\varrho(\omega, \eta), |x-y|)$.

Any string $w = w_1w_2 \dots w_k$ of elements from $\{1, 2, \ldots , n\}$ is called a \emph{block} or a \emph{word} of length~$k$.
Concatenation of blocks is indicated by juxtaposition; if $u=u_1\dots u_m$ and $v=v_1\dots v_n$ then $uv=u_1\dots u_mv_1\dots v_n$. We will also use the notation $w^n = w\dots w$ ($n$-times) and $w^\infty = www \dots$. Clearly, a point $\alpha \in \Sigma^{+}_{n}$ is periodic for the shift transformation if and only if $\alpha = (\alpha_0\alpha_1\ldots\alpha_{p-1})^\infty$ for some block $\alpha_0\alpha_1\ldots\alpha_{p-1}$.

If $B\subseteq \Sigma^{+}_{n}$ is nonempty, closed and invariant, i.e. $S(B)\subseteq B$, then $B$ with $S$ restricted to $B$ is again a dynamical system, called a \emph{subshift}. If no confusion can arise, the phase space $B$ itself is called a subshift. If $w_0w_1w_2\dots$ is an element of the subshift $B$, then every word
$w_nw_{n+1}\dots w_{n+k}$ ($n,k\geq0$) is called a \emph{$B$-word}.

\section{Finite nonautonomous systems and non-shrinking of intervals}\label{S:non-shrink}

The following theorem will be used in the proof of our main result. 

\begin{theorem}\label{T:non-shrink}
	Let finitely many expanding, piecewise monotone, continuous maps $T_1, T_2,\ldots, T_n$ from $[0, 1]$ to $[0, 1]$ be given. Then for every $\varepsilon > 0$ there is $\gamma >0$ such that for any finite nonautonomous system  $(f_i)^{\infty}_{i=0}$ containing only maps from the given finite family, and for any interval $U$ with $|U|\geq\varepsilon$ we have $\inf_{i\geq 0}|f^i_0(U)| \geq \gamma$.
\end{theorem}

\begin{proof}
	Fix $\varepsilon>0$. We are going to choose a suitable $\gamma>0$. For $1\leq j\leq n$ let $\alpha_j > 1$ be the expansion rate of the expanding map $T_j$. Set $\alpha=\min(\alpha_1, \alpha_2,\ldots, \alpha_n)$. We have $\alpha>1$. We fix a positive integer $m$ such that  $\alpha^m>2$. 
	
	Set $\mathcal{F}=\{ T_1, T_2,\ldots, T_n\}^m$. For any $\Delta=(H_1, H_2,\ldots, H_m)\in \mathcal{F}$ we find a number $\beta_{\Delta}>0$ such that any interval $U\subseteq [0, 1]$ of length $\leq\beta_{\Delta}$, which has a critical point of~$H_1$ as endpoint, satisfies the condition that
	
	\begin{quote}
		the interval $G_{j-1}(U)$ has no critical point of $H_j$ in its interior for $1\leq j\leq m$,
	\end{quote}
	
	\noindent where $G_0$ is the identity and $G_i=H_i\circ\cdots\circ H_2\circ H_1$ for $1\leq i\leq m$. Now we set $\beta=\min_{\Delta\in \mathcal{F}}\beta_\Delta$. Since $\mathcal F$ is finite we have $\beta>0$. Finally we choose $\gamma=\min(\frac{\varepsilon}{2}, \beta)$.
	
	Now let  $(f_i)^{\infty}_{i=0}$ be a finite nonautonomous system consisting of the maps $T_1, T_2,\ldots, T_n$ and let $U$ be an interval with $|U|\geq\varepsilon$. We have to show that $|f^i_0(U)|\geq\gamma$ for all $i\geq 0$.
	
	Let $j_1\geq 0$ be minimal such that the interval $f^{j_1}_0(U)$ contains a critical point of $f_{j_1}$ in its interior. Since $f^i_0(U)$ does not contain a critical point of $f_i$ in its interior for $0\leq i<j_1$ and all maps are expanding, we get 
	\[
	|f^i_0(U)| \geq \varepsilon \geq \gamma \quad \text{for} \quad 0\leq i\leq j_1.
	\]
	The interval $f^{j_1}_0(U)$ has length $\geq \varepsilon$ and contains a critical point of $f_{j_1}$ in its interior. Because of $\gamma \leq \frac{\varepsilon}{2}$ we find an interval $V_1\subseteq f^{j_1}_0(U)$, which has a critical point of $f_{j_1}$ as endpoint and satisfies $|V_1|=\gamma$. Let $j_2 \geq 0$ be minimal such that the interval $f^{j_2}_{j_1} (V_1)$ contains a critical point of $f_{s_2}$ in its interior, where $s_2 = j_1 + j_2$. Since $f^i_{j_1}(V_1)$ does not contain a critical point of $f_{j_1+i}$ in its interior for $0\leq i < j_2$ and all maps are expanding with expansion rate at least $\alpha$, we get 
	\[
	|f^{j_1+i}_0 (U)| \geq |f^i_{j_1}(V_1)| \geq \alpha^i \gamma \geq \gamma \quad 
	\text{for} \quad 0 \leq i \leq j_2.
	\]
	By the choice of $\beta$ and because of $\gamma \leq \beta$ we get $j_2 \geq m$. This implies $|f^{j_2}_{j_1}(V_1)| \geq \alpha^m\gamma>2\gamma$. Since the interval $f^{j_2}_{j_1}(V_1)$ contains a critical point of $f_{s_2}$ in its interior, we find an interval $V_2 \subseteq f^{j_2}_{j_1}(V_1)$, which has a critical point of $f_{s_2}$ as endpoint and satisfies $|V_2|=\gamma$. Now we can continue as above. Let $j_3\geq0$ be minimal such that the interval $f^{j_3}_{s_2}(V_2)$ contains a critical point of $f_{s_3}$ in its interior, where $s_3=s_2+j_3$. Since $f^i_{s_2}(V_2)$ does not contain a critical point of $f_{s_2+i}$ in its interior for $0\leq i < j_3$, we get as above
	\[
	|f^{s_2+i}_0(U)| \geq |f^i_{s_2}(V_2)| \geq \alpha^i \gamma \geq \gamma \quad
	\text{for} \quad 0\leq i \leq j_3.
	\]
	Again we get $j_3\geq m$ and $|f^{j_3}_{s_2}(V_2)| \geq \alpha^m \gamma > 2 \gamma$. And we find an interval $V_3 \subseteq f^{j_3}_{s_2}(V_2)$, which has a critical point of $f_{s_3}$ as endpoint and satisfies $|V_3|=\gamma$.
	
	Repeating this procedure we finally get $|f^i_0(U)|\geq \gamma$ for all $i\geq 0$. Therefore, the theorem is proved. 
\end{proof}

In Theorem~\ref{T:non-shrink} we assume that the piecewise monotone maps $T_1,\dots, T_n$ are expanding, i.e. they have expansion rates $\alpha_j>1$, $j=1,\dots, n$. We are interested in whether it works under the weaker assumption that $|T_j(x)-T_j(y)| \geq |x-y|$ holds for every $j=1, \dots, n$ and for all $x, y$ which are in the same interval of monotonicity of $T_j$. The next example shows that this is not the case, even if the maps are surjective and piecewise linear.

\begin{example}[Theorem~\ref{T:non-shrink} does not work if $\alpha_j\geq 1$ instead of $\alpha_j>1$, $j=1,\dots,n$]\label{Ex:nonaut}
	We are going to find a finite nonautonomous system $(f_i)_{i=0}^\infty$ on $[0,1]$ made of three piecewise linear, surjective, continuous maps $\varphi$, $f$ and $g$ with slopes (in absolute value) $\geq 1$ such that for some non\-de\-ge\-ne\-ra\-te interval~$J$, $\lim_{n\to \infty} |f_0^n(J)| = 0$. 
	
	We will use the following notation for ``connect the dots" maps. Let $0 = a_0 < a_1 < \dots < a_n = 1$ and $b_i \in I = [0,1]$, $i=0,1,\dots, n$. Then by $\langle (a_0, b_0), \dots, (a_n, b_n) \rangle$ we will denote the map $I\to I$ which sends $a_i$ to $b_i$, $i=0,1,\dots, n$ and is linear on each of the intervals $[a_i, a_{i+1}]$, $i=0,1,\dots, n-1$.
	
	To construct the system, fix a small positive irrational number $\xi$; in fact any irrational $\xi$ with $0< \xi < 1/4$ is good for our purposes. Consider the following piecewise linear maps:
	\begin{align}
	\varphi  & = \langle (0, \xi), (1-\xi, 1), (1, 0) \rangle,  \notag \\
	f  & = \langle (0, 1), (1-2\xi, 0), (1, 2\xi) \rangle,   \notag \\
	g  & = \langle (0, 1), (1/4, 1/4), (1/2, 0), (1, 1/2) \rangle~.   \notag
	\end{align}
	All the slopes of $\varphi$, $f$ and $g$ are (in absolute value) $\geq 1$.
	Put $J=[0,\xi]$. The system $(f_i)_{i=0}^\infty$ will have the form
	\begin{equation}\label{Eq:nonaut}
	(f_i)_{i=0}^\infty \, = \, \underbrace{\varphi, \varphi, \dots, \varphi}_{k_1}, \, \psi_1, \, \underbrace{\varphi, \varphi, \dots, \varphi}_{k_2}, \, \psi_2, \, \underbrace{\varphi, \varphi, \dots, \varphi}_{k_3}, \, \psi_3, \, \dots
	\end{equation}
	where each $\psi_i$ is either $f$ or $g$. We choose
	\[
	k_1 = \min \{k: \varphi^k(J) \text{ contains $1/2$ in its closed middle third or intersects } [1-\xi,1]\}~.
	\]
	Note that the interval $\varphi^{k_1}(J)$ is obtained from $J$ by translation by $k_1\xi$ and so has length $\xi$.
	If $\varphi^{k_1}(J)$ contains $1/2$ in its closed middle third, we choose $\psi_1 = g$ and then $\psi_1(\varphi^{k_1}(J))$ is an interval whose left endpoint is $0$ and has length $\leq (2/3)\xi = (2/3)|J|$. Otherwise the interval $\varphi^{k_1}(J)$ intersects $[1-\xi,1]$ and thus it is a subset of $[1-2\xi,1]$. We choose $\psi_1=f$. Then $\psi_1(\varphi^{k_1}(J))$ is a subinterval of $[0,2\xi]$ and has length $\xi$; it is obtained from $\varphi^{k_1}(J)$ by translation by $2\xi-1$.
	
	In either case we have that $J_1=\psi_1(\varphi^{k_1}(J))$ is a subinterval of $[0,2\xi]$ and we choose
	\[
	k_2 = \min \{k: \varphi^k(J_1) \text{ contains $1/2$ in its closed middle third or intersects } [1-\xi,1]\}~.
	\]
	Again, if $\varphi^{k_2}(J_1)$ contains $1/2$ in its closed middle third, we choose $\psi_2 = g$ and then
	\[
	\left | \psi_2(\varphi^{k_2}(J_1))\right | \leq \frac{2}{3} \left |\varphi^{k_2}(J_1) \right | = \frac{2}{3} \left |J_1 \right |.
	\]
	Otherwise the interval $\varphi^{k_2}(J_1)$ intersects $[1-\xi,1]$ and we choose $\psi_2=f$. Then the interval $\psi_2(\varphi^{k_2}(J_1)) \subseteq [0,2\xi]$, obtained from $\varphi^{k_2}(J_1)$ by translation by $2\xi-1$, has the same length as $J_1$.
	We continue in the same way with the interval $J_2=\psi_2(\varphi^{k_2}(J_1))$, etc. By induction we get the system~(\ref{Eq:nonaut}). 
	
	Note that when we iterate the interval $J=[0,\xi]$ under this system, its length does not change except of the moments when we apply $\psi_i = g$, when the length decreases at least by factor $2/3$. We claim that $\psi_i = g$ for infinitely many $i$'s, which immediately implies that $\lim_{n\to \infty} |f_0^n(J)| = 0$. 
	
	Suppose, on the contrary, that $\psi_i = g$ only for finitely many $i$'s. Then for large enough $m$ the interval $J_m$ (which is a subset of $[0, 2\xi]$ and has length at most $\xi$) has the following property: We have $J_m = f_0^n(J)$ for some $n=n(m)$ and the middle thirds of the (nondegenerate) intervals
	\begin{equation}\label{Eq:Jm etc}
	J_m = f_0^n(J), \, f_0^{n+1}(J),\, f_0^{n+2}(J), \, \dots
	\end{equation}
	do not contain $1/2$. Therefore each of the intervals $f_0^{n+i}(J)$, $i=1,2,\dots$, is obtained from the previous interval in \eqref{Eq:Jm etc} by translation, either by $\xi$ or by $2\xi -1$. By identifying $0$ and $1$ in $[0,1]$, we get a circle. Then our iterative process, assigning the sequence \eqref{Eq:Jm etc} to $J_m$, corresponds to the iterating $J_m$ by the irrational circle rotation by the angle $2\pi \xi$ except of the moments when our interval is $\xi$-close to $1$ ``from the left", when we use the rotation by $2\pi \cdot 2\xi$. However, the rotation by $2\pi \cdot 2\xi$ is the same as to use the rotation by $2\pi \xi$ twice. Since $\xi \in (0,1/4)$, during the first of these two applications of the rotation by $2\pi \xi$ our interval does not hit $1/2$. Therefore, using the fact that the middle thirds of the intervals in \eqref{Eq:Jm etc} do not contain $1/2$, we conclude that  the middle third of $J_m$ never hits $1/2$ under the circle rotation by the angle $2\pi \xi$. This contradicts the irrationality of $\xi$ (irrational rotations are minimal). 
\end{example}	
	
\section{Main result}\label{S:main}

Our main result is the following sufficient condition for the step skew product~\eqref{Eq:F} to have the specification property.

\begin{theorem} \label{T:main}
	Let $T_1, T_2, \ldots, T_n$ be piecewise monotone, continuous maps on $[0,1]$, which are expanding and surjective. Suppose that $B\subseteq \Sigma^+_n$ is a subshift which has the specification property and contains a periodic point $\alpha = (\alpha_0\alpha_1\ldots\alpha_{p-1})^\infty$ such that $T_{\alpha_{p-1}} \circ \cdots \circ T_{\alpha_1} \circ T_{\alpha_0}$ is mixing. Let $F:X\to X$ be the step skew product with $X=B\times [0,1]$ and $F(\omega, x) = (S(\omega), T_{\omega}(x))$, where $S\colon B\to B$ is the shift transformation and  $T_{\omega} = T_q$, if $q$ is the first symbol of $\omega$. Then $(X, F)$ has the specification property.
\end{theorem}

\begin{proof} 
	Fix $\varepsilon \in (0,1)$. We have to show that there is an integer $M$ which has the following property: For any integer $k\geq 2$, and for any 
	\[
	(\omega_1, x_1,n_1),  (\omega_2, x_2, n_2),\dots , (\omega_k, x_k, n_k) \in B \times [0,1] \times \mathbb {N}
	\]
	there exists a point $(\eta, z) \in B \times [0,1]$ with $F^{r_k}(\eta, z) = (\eta, z)$ such that
	\[
	d(F^i(\omega_j, x_j), F^{r_{j-1}+i}(\eta, z)) \leq \varepsilon \quad \text{for } \, 0\leq i <n_j \, \text{ and } \, 1\leq j \leq k
	\]
	where $r_0=0$ and $r_j= n_1+n_2+\cdots+n_j+jM$ for $1\leq j\leq k$.
	
	In order to show this, we first choose $M$, which depends only on the fixed $\varepsilon>0$. For the maps $T_1, T_2, \ldots, T_n$ and the fixed $\varepsilon$ let $\gamma>0$ be as in Theorem~\ref{T:non-shrink}. For this $\gamma$ and for the transformation $T_{\alpha_{p-1}}\circ \cdots \circ T_{\alpha_1} \circ T_{\alpha_0}$ choose $m$ as in Lemma~\ref{L:mix-exp}. For the given $\varepsilon$ let $K$ be the gap length of the dynamical system $(B,S)$ which has the specification property. Then set $M=mp+2K$.

	Since the dynamical system $(B,S)$ has specification property (with $K$ being the gap length for the given $\varepsilon$), for 
	\[
	(\omega_1, n_1), (\alpha, mp), (\omega_2, n_2), (\alpha, mp), \dots , (\omega_k, n_k), (\alpha, mp) \in B \times \mathbb {N}
	\]
	there exists a point $\eta \in B$ such that
	\begin{equation}\label{Eq:1eta}
	S^{r_k}(\eta) = \eta \quad \text{and} \quad \varrho(S^i(\omega_j), S^{r_{j-1}+i}(\eta)) \leq \varepsilon \quad \text{for } \, 0\leq i<n_j \, \text{ and }  \, 1\leq j \leq k
	\end{equation}
	where $r_0 =0$ and $r_j = n_1+n_2+\cdots+n_j+jM$ for $1\leq j \leq k$ are as above, and such that
	\begin{equation}\label{Eq:eta}
	\varrho(S^i(\alpha), S^{r_{j-1}+n_j+K+i}(\eta)) \leq \varepsilon \quad \text{for } \, 0\leq i<mp \, \text{ and } \, 1\leq j \leq k. 
	\end{equation}
	Therefore $\eta \in B$ is already found with the desired properties.
	
	Now let $(f_i)^{\infty}_{i=0}$ be the nonautonomous system determined by this point $\eta \in \Sigma^+_n$, this means $f_i = T_q$, if $q$ is the symbol in the $i$-th place in $\eta$ (we start counting with 0). Remember that for $i\geq1$ and $n\geq0$ we set $f^i_n= f_{n+i-1} \circ \cdots \circ f_{n+1} \circ f_n $ and $f^0_n$ denotes the identity. Now set 
	\[
	\widetilde{J_j} =\{y\in[0,1]:|f^i_{r_{j-1}} (x_j) - f^i_{r_{j-1}}(y)| \leq \varepsilon \, \text{ for } 
	\, 0\leq i <n_j\} \quad \text{for } \, 1 \leq j \leq k.
	\]
	Let $J_j$ be the connected component of $\widetilde{J_j}$ which contains the point $x_j$. 
	
	Since the maps $f_{n}^{i}$ are continuous, for every $1\le j \le k$ we get that the set $J_{j}$ is a closed interval which is a neighbourhood of $x_j$ (in the topology of $[0,1]$). Furthermore, for $1\le j \le k$ there is $ i<n_{j}$ with $|f_{r_{j-1}}^{i}(J_{j})|\ge\varepsilon$ (otherwise, by continuity, an interval properly containing $J_j$ would be a component of $\widetilde{J_j}$, a contradiction). In the following we set $s_{j}=n_{j}+M$ for $1\le j \le k$. We have then $ r_{0}=0$ and $r_{j}= r_{j-1}+s_{j}$ for $1 \le j \le k$. 
	
	We start with $ J_{1}$. Because of $|f_{0}^{i}(J_{1})| \ge  \varepsilon $ for some 
	$i < n_{1}$ we have \mbox{$|f_{0}^{n_{1}+K}(J_{1})| \ge  \gamma$} by Theorem~\ref{T:non-shrink}. 
	By~\eqref{Eq:eta} and~\eqref{Eq:metric} we have $f_{n_{1}+K}^{mp}=U^{m}$ with $U=T_{\alpha_{p-1}} \circ \cdots \circ T_{\alpha_1} \circ T_{\alpha_0}$. 
	By Lemma~\ref{L:mix-exp} we get then $f_{n_{1}+K}^{mp} \circ f_{0}^{n_{1}+K}(J_{1})=[0, 1]$, 
	which means \mbox{$f_{0}^{n_{1}+K+mp}(J_{1})=[0, 1]$}. 
	Since the fibre maps are surjective and we have $n_{1}+ K + mp < n_{1} + M = s_{1}$, 
	this implies $f_{0}^{s_{1}}(J_{1})=[0, 1]$.
	We find an interval $K_{1} \subseteq J_{1} $ with $f_{0}^{r_{1}}(K_{1})=J_{2}$,
	where we have used that $s_{1}=r_{1}$.
	
	Now we consider $J_{2}$. In the same way as above we get $f_{r_{1}}^{s_{2}}(J_{2})=[0, 1]$. 
	Because of $f_{r_{1}}^{s_{2}} \circ f_{0}^{r_{1}}=f_{0}^{r_{2}} $ this implies 
	$f_{0}^{r_{2}}(K_{1})=[0, 1]$.  We find an interval $K_{2} \subseteq K_{1} \subseteq J_{1}$ with  $f_{0}^{r_{2}}(K_{2})=J_{3}$.
	
	Next we consider $J_{3}$. As above we get $f_{r_{2}}^{s_{3}}(J_{3})= [0, 1]$. 
	Because of $f_{r_{2}}^{s_{3}} \circ f_{0}^{r_{2}}=f_{0}^{r_{3}} $ this implies 
	$f_{0}^{r_{3}}(K_{2})= [0, 1]$. We find an interval $K_{3} \subseteq K_{2} \subseteq K_1 \subseteq J_{1}$ with $f_{0}^{r_{3}}(K_{3})=J_{4}$. 
	
	Finally we end with intervals $K_{k-1} \subseteq K_{k-2}  \subseteq  \dots \subseteq K_{1} \subseteq J_{1}$
	satisfying $f_{0}^{r_{k}}(K_{k-1})=[0, 1]$ and $f_{0}^{r_{j-1}}(K_{j-1})=J_{j}$ for $2 \le j \le k$. 
	Because of $f_{0}^{r_{k}}(K_{k-1})= [0, 1]$ we find a point $z \in K_{k-1} $ with 
	$f_{0}^{r_{k}}(z) = z$. 
	Furthermore, we have $ z \in J_{1} $ and for $2 \le j \le k$ we have $f_{0}^{r_{j-1}}(z) \in J_{j}$,
	since $z \in K_{j-1}$ and $f_{0}^{r_{j-1}}(K_{j-1})=J_{j}$. 
	By the definition of the sets $J_{j}$ this implies  
	\[
	|f_{r_{j-1}}^{i}(x_{j})-f_{0}^{r_{j-1}+i}(z)| \le \varepsilon \quad \text{for} \quad 0 \le i < n_{j} \, \text{ and } \, 1 \le j \le k.
	\] 
	Together with~\eqref{Eq:1eta} this gives using the definition of the nonautonomous system 
	$(f_{i})_{i=0}^{\infty}$ that $F^{r_k}(\eta,z)=(\eta,z)$ and that
	\[
	d(F^{i}(\omega_{j},x_{j}), F^{r_{j-1}+i}(\eta,z))\le \varepsilon \quad \text{for} \quad 0 \le i < n_{j} \, \text{ and } \, 1 \le j \le k.
	\]
	The theorem is proved. 
\end{proof}

\begin{corollary}\label{C:cor}
	Let $T_1,T_2,\ldots, T_n$ be piecewise monotone, continuous maps on $[0, 1]$, 
	which are expanding and mixing. Suppose that $B\subseteq \Sigma^+_n$ is a subshift which has a fixed point. Let $F:X\to X$ be the corresponding step skew product as in  Theorem~\ref{T:main}. Then $(X, F)$ has the specification property if and only if the subshift $(B,S)$ has the specification property.
\end{corollary}

\begin{proof} 
	One implication is trivial, because the specification property is preserved by passing to a factor. Now assume that the subshift $(B, S)$ has the specification property. Since we assume that the (piecewise monotone) fibre maps are mixing (hence, they are also surjective) and that $S$ has a fixed point in $B$, all the assumptions of Theorem~\ref{T:main} are fulfilled. Thus $(X, F)$ has the specification property.
\end{proof}

\begin{remark}
	Recall that, by~\cite{Ber}, a subshift $B$ has the specification property if and only if it has a \emph{uniform transition length}, meaning that there exists a positive integer $M$ such that for any $B$-words $u$ and $v$ there exists a $B$-word $w$ of length $M$ such that $uwv$ is a $B$-word.
	
	For subshifts of finite type,  the specification property is equivalent to mixing, see~\cite{Bow} or \cite[Proposition 21.2 and Proposition 21.3]{DGS}. For the specification property of sofic systems see~\cite{W}.
\end{remark}

{\bf Acknowledgement.} The author thanks Franz Hofbauer for numerous useful discussions. This paper would not exist without him. 
He also thanks Piotr Oprocha for a discussion on equivalent definitions of the specification property.

\end{large}
\end{document}